\documentclass[12pt]{amsart}

\usepackage{amsmath, amssymb, amscd, newlfont}

\newcommand{\bbQ}{{\mathbb{Q}}}

\newcommand{\bbP}{{\mathbb{P}}}

\newcommand{\bbR}{{\mathbb{R}}}
\newcommand{\bbZ}{{\mathbb{Z}}}

\newcommand{\ord}{{\mathrm{ord}}}

\newcommand{\supp}{{\mathrm{supp}}}

\numberwithin{equation}{section}

\newtheorem{Lem}[equation]{Lemma}

\newtheorem{Thm}[equation] {Theorem}
\newtheorem{Cor}[equation]{Corollary}

\title
[Models  of $X_0(N)$]
{Integral Models   of  $X_0(N)$ and Their Degrees}

\author{Goran Mui\'c }
\address{
Department of Mathematics, 
University of Zagreb,
Bijeni\v cka 30, 10000 Zagreb,
Croatia}
 \email{gmuic@math.hr}

\begin{document}
\begin{abstract}
In this paper we compute the degree of a curve which is the image of a mapping 
$z\longmapsto (f(z): g(z): h(z))$ constructed 
out of three linearly independent modular forms of the same even weight $\ge 4$ into $\mathbb P^2$. 
We  prove that in most cases this map is a birational equivalence and defined over $\mathbb Z$. 
We use this to construct models of $X_0(N)$,
$N\ge 2$,  using modular forms in $M_{12}(\Gamma_0(N))$ with integral $q$--expansion. The models have degree equal to 
$\psi(N)$ (a classical Dedekind psi function). When genus is at least one, we show the existence of 
models constructed using cuspidal forms in $S_4(\Gamma_0(N))$ of degree $\le \psi(N)/3$
and in $S_6(\Gamma_0(11))$ of degree $4$.
As an example of a different kind, we compute the  formula for the total degree i.e., the degree considered as a 
polynomial 
of two (independent) variables of the classical modular polynomial (or the degree of the canonical model of 
$X_0(N)$).
\end{abstract}
\maketitle

\maketitle
\section{Introduction}

We start this paper by a classical example in order to motivate further results.
Let $N\ge 2$. The modular curves $X_0(N)$ have canonical plane models constructed by Hauptmoduln $j$ \cite{shi}. 
More precisely, its function field over $\mathbb C$ is generated by $j$ and 
$j(N\cdot )$. The classical modular polynomial $\Phi_N\in \mathbb Z[x, y]$ is 
the minimal polynomial of $j(N\cdot )$ over $\mathbb C(j)$. It is symmetric in $x$ and $y$ i.e.,
$\Phi_N(x, y)=\Phi_N(y, x)$ and it has a degree in $x$ or $y$ equal to the Dedekind psi function
$$
\psi(N)=N\prod_{p |N} \left(1+ \frac{1}{p}\right).
$$
It is also irreducible as a polynomial in $x$ as an element of 
$\mathbb C(y)[x]$ or in $y$  as an element of 
$\mathbb C(x)[y] $. But then it is irreducible in $\mathbb C[x, y]$. 
This polynomial is rather mysterious and difficult to compute \cite{BKS}.  
A classical result is a formula for the degree of $\Phi_N(x, x)$ (see \cite{cox}, Proposition 13.8):
$$
2\sum_{\substack{k|N\\N\ge k> \sqrt{N}}}\frac{k}{(k, N/k)}\varphi((k, N/k))+\varphi(\sqrt{N}),
$$
where where $\varphi$ is the Euler function, and $(\ , \ )$ denotes the greatest common 
divisor. We let $\varphi(\sqrt{N})=0$ if $N$ is not a perfect square.

As a simple consequence of our general result (see Theorem \ref{ithm-2} below), we prove the following formula for the degree 
of $\Phi_N$ considered as a polynomial of two variables $x$ and $y$:

\begin{Thm}\label{ithm} Let $N\ge 2$. Then, the total degree of $\Phi_N(x, y)$ is equal to the degree of 
 $\Phi_N(x, x)$.
\end{Thm}

Having in mind the shapes of polynomials $\Phi_N(x, y)$ (say, computed using MAGMA computer system), this result is not 
very surprising. It is quite likely that classical methods \cite{cox} would give the proof of Theorem \ref{ithm} 
also. Our approach to this example is from a very general theorem (see Theorem \ref{ithm-2} below) which can be used in 
many other cases as we explain in the paper. It would also be interesting to apply the approach to other known cases. 

\vskip.2in
Let us explain the proof of Theorem \ref{ithm}.  Let 
$$
E_4(z)=1+240 \sum_{n=1}^\infty \sigma_3(n)q^n$$
 be the usual Eisenstein series, 
and let 
$$
\Delta(z)=q+\sum_{n=2}\tau(n)q^n
$$
 be the Ramanujan delta function. Then,   $j=E^3_4/\Delta$. 
Since  $E^3_4(N\cdot ), \Delta (N\cdot ) \in M_{12}(\Gamma_0(N))$,
the fact that the function field over $\mathbb C$ is generated by  $j$ and 
$j(N\cdot )$ means that the holomorphic map

\begin{equation}\label{mmap}
\mathfrak a_z\mapsto (1: j(z): j(Nz))=
(\Delta(z)\Delta(Nz): E^3_4(z)\Delta(Nz): E^3_4(Nz) \Delta(z))
\end{equation}
is a birational uniformization by $S_{24}(\Gamma_0(N))$.  The image, say $\cal C$, is an irreducible 
projective curve (Lemma  \ref{p2-1}). Since $\Phi_n$ is irreducible in $\mathbb C[x, y]$, the  homogenization
$$
x_0^{\deg{\Phi_N}}\Phi(x_1/x_0, x_2/x_0)
$$
is an irreducible homogeneous polynomial of total degree $\deg{\Phi_N}$. Its locus is $\cal C$.  
By definition, the degree of $\cal C$ is the total degree of this irreducible polynomial. 
 A classical geometric interpretation of the degree of 
$\cal C$  is the number of intersection of $\cal C$ with a generic line $l\subset \mathbb P^2$. 
We use this fact along with the following general Theorem \ref{ithm-2} to prove  Theorem \ref{ithm}.

We start more generally from any Fuchsian group of the first kind $\Gamma$ (see \cite{Miyake}).
For an even integer $m\ge 4$,  we denote the space cuspidal forms by $S_m(\Gamma)$ and the space of all 
modular forms by $M_m(\Gamma)$  of weight $m$. Let $g(\Gamma)$ be the genus of $\mathfrak R_\Gamma$. 

We select three linearly independent modular forms $f, g$, and $h$ in $M_m(\Gamma)$, and  
construct the holomorphic map 
$\mathfrak R_\Gamma\longrightarrow \Bbb P^{2}$ given by
$$
\mathfrak a_z\longmapsto (f(z):  g(z): h(z)).
$$
The image is an irreducible projective plane curve   which we denote by 
$\cal C(f, g, h)$ of degree $\le \dim M_m(\Gamma)+g(\Gamma)-1$. 
The details about this map can be found in the proof of 
Lemma  \ref{p2-1}. If the curve $\cal C(f, g, h)$ has no singularities, 
then it is itself a compact Riemann surface, and one can define the degree $\deg{(\varphi)}$ of the map $\varphi$ 
as usual (a number of preimages of a point counted with multiplicities). But $\cal C(f, g, h)$ is rarely non-singular, 
and one needs to modify this standard definition (in the theory of compact Riemann surfaces) 
in order get the correct definition of $\deg{(\varphi)}$. This can be extracted from the general intersection theory 
\cite{fulton} (and it is known since 19th century in many forms, 
analytic and algebraic, proved to be equivalent by Fulton), 
but we want to restrict ourselves to more elementary tools so  we supply  short and simple direct 
argument (see Lemma \ref{p2-3}) for the benefit of the reader who  might not be very versed in the Intersection theory. 
Here, $\deg{(\varphi)}$ is the number of  
preimages of a non--singular point  counted with multiplicities. 
This definition is the one that behaves well as we 
explain in the proof of the main theorem  (see Lemma \ref{p2-6}). 
For $f\in M_m(\Gamma)$, $f\neq 0$, we attach an integral effective divisor  
$\mathfrak c'_{f}$ by subtracting from its rational effective divisor $\text{div}{(f)}$ necessary contribution at 
elliptic points (see Lemma \ref{prelim-1} (vi)). The degree of this divisor is $\dim M_m(\Gamma)+ g(\Gamma)-1$. 
In addition, if $f\in S_m(\Gamma)$, then we  subtract necessary contribution from $\mathfrak c'_{f}$ at cusps, and get a 
divisor  $\mathfrak c_{f}$ (see (\ref{2div})) which has degree  $\dim S_m(\Gamma)+ g(\Gamma)-1$.
The first main result of the paper is the following theorem (see Theorem \ref{p2-120}).

\begin{Thm}\label{ithm-2} Assume that  $m\ge 4$ is an even integer such that 
$\dim M_m(\Gamma)\ge 3$. Let $f, g, h\in M_m(\Gamma)$ be linearly independent.
Then, we have the following:
$$
\deg{(\varphi)}\cdot \deg{C(f, g, h)}=\dim M_m(\Gamma) + g(\Gamma) -1 - \sum_{\mathfrak a\in \mathfrak R_\Gamma} 
\min{\left(\mathfrak c'_{f}(\mathfrak a),  \mathfrak c'_{g}(\mathfrak a), 
\mathfrak c'_{h}(\mathfrak a) \right)}.
$$
Moreover, if $f, g, h\in S_m(\Gamma)$, then 
$$
\deg{(\varphi)}\cdot \deg{C(f, g, h)}=\dim S_m(\Gamma) + g(\Gamma) -1 - \sum_{\mathfrak a\in \mathfrak R_\Gamma} 
\min{\left(\mathfrak c_{f}(\mathfrak a),  \mathfrak c_{g}(\mathfrak a), 
\mathfrak c_{h}(\mathfrak a) \right)}.
$$
In particular, if $\varphi$ is birational, then $\deg{(\varphi)}=1$ (since it is generically injective), and we have 
in either case a simple formula for the  degree of $\deg{C(f, g, h)}$ in terms of $f, g$, and $h$.
\end{Thm}

The proof of the theorem uses fine points of the theory of compact Riemann surfaces and algebraic curves. 
Preliminaries about divisors attached to cuspidal 
modular forms are stated in Section \ref{prelim}. The proof of the theorem is given in Section \ref{p2} in a 
series of Lemmas. In Section \ref{appmod} we prove Theorem \ref{ithm} using Theorem \ref{ithm-2}.

As we see, the canonical model of $X_0(N)$ is obtained using cusp forms of weight $24$. 
In Section \ref{gen}, we show that 
this model is just one in the series of birational models constructed using cusp forms (this is not essential). 
To state the main result of Section \ref{gen} (see Theorem \ref{p2-thm}), we introduce some notation. 

Let $m\ge 4$ be an even integer and let $W\subset M_m(\Gamma)$ be a non--zero linear subspace. 
Then,  we say that  $W$ generates the field of rational functions $\mathbb C(\mathfrak R_\Gamma)$ if $\dim W\ge 2$, and 
there exists  a basis
$f_0, \ldots, f_{s-1}$ of $W$,  such that the the holomorphic map  $\mathfrak R_\Gamma\longrightarrow \bbP^{s-1}$
given by $\mathfrak a_z\mapsto \left(f_0(z): \cdots : f_{s-1}(z)\right)$ is birational. Clearly, this notion does not 
depend on the choice of the basis used. It is also obvious that this is equivalent with the fact that
$\mathbb C(\mathfrak R_\Gamma)$ is generated over $\mathbb C$ with the quotients $f_i/f_0$, $1\le i\le s-1$.
We need one more defintion. We say that $W\not\subset S_m(\Gamma)$ (resp., $W\subset S_m(\Gamma)$)
separates the points on $\mathfrak R_\Gamma$ if for each $\mathfrak a\in 
\mathfrak R_\Gamma$ there exists $f\in W$, $f\ne 0$, such that $\mathfrak c'_f(\mathfrak a)=0$
(resp., $\mathfrak c_f(\mathfrak a)=0$). The geometric meaning of 
this assumption is that the complete linear system attached to the divisor of above holomorphic map into $\bbP^{s-1}$
has no base points (see the proof of Lemma \ref{prelim-8-cor}). Now, we are ready to state the main result of 
Section \ref{gen} (see Theorem \ref{p2-thm}).

\begin{Thm}\label{ithm-3}
Assume that  $m\ge 4$ is an even integer. Let $W\subset M_m(\Gamma)$, $\dim W\ge 3$,  be a subspace which generates
 the field of rational 
functions $\mathbb C(\mathfrak R_\Gamma)$, and separates the points on $\mathfrak R_\Gamma$. For example, if 
$\dim S_m(\Gamma) \ge \max{(g(\Gamma)+2, 3)}$, then we can take $W=S_m(\Gamma)$. Let $f, g \in W$ be linearly independent.
Then there exists  a non--empty Zariski open set $\cal U\subset W$ such that for any 
$h\in \cal U$ we have the following:
\begin{itemize}
\item[(i)] $\mathfrak R_\Gamma$ is birationally equivalent to  $\cal C(f, g, h)$, and 
\item[(ii)]  $\cal C(f, g, h)$ has degree equal to $\dim M_m(\Gamma)+g(\Gamma)-1$ (resp., 
$\dim S_m(\Gamma)+g(\Gamma)-1$) if $W\not\subset S_m(\Gamma)$ (resp., $W\subset S_m(\Gamma)$).
\end{itemize}
\end{Thm}

In Section \ref{exist} we prove the following corollary (see Corollary \ref{exist-1}):

\begin{Cor}
Assume that $N\not\in \{1, 2, 3, 4, 5, 6, 7, 8, 9, 10, 12, 13, 16, 18, 25\}$ (so that the genus 
$g(\Gamma_0(N))\ge 1$). Assume that  $m\ge 4$ (if $N\neq 11$) and $m\ge 6$ (if $N=11$) is an even integer. 
Let $f, g \in S_m(\Gamma_0(N))$  be linearly independent with integral $q$--expansions.
Then, there exists infinitely many   $h\in S_m(\Gamma_0(N))$   
with integral $q$--expansion such that  we have the following:
\begin{itemize}
\item[(i)] $X_0(N)$ is birationally equivalent to  $\cal C(f, g, h)$, 
\item[(ii)] $\cal C(f, g, h)$ has degree equal to $\dim S_m(\Gamma_0(N))+g(\Gamma_0(N))-1$ 
(this number can be easily explicitly computed using Lemma \ref{prelim-1}(v) and (\ref{gggenus})); 
if $N=11$, then the minimal possible degree achieved (for $m=6$) is $4$, and if $N\neq 11$, then the minimal 
possible degree achieved (for $m=4$) is
$$
\frac13 \psi(N)-  \frac13 \nu_3 -\sum_{d>0, d|N} \phi((d, N/d)),
$$
where $\nu_3$ is the number of elliptic elements of order three on $X_0(N)$,
$\nu_3=0$ if $9|N$, and $\nu_3=\prod_{p|N}\left(1+ \left(\frac{-3}{p}\right)\right)$, otherwise.

\item[(iii)] the equation of $\cal C(f, g, h)$ has integral coefficients.
\end{itemize}
\end{Cor}

\vskip .2in
As an example to Theorem \ref{ithm-3}, we consider the subspace $W\subset M_{12}(\Gamma_0(N))$ which basis is 
$\Delta, E^3_4, \Delta(N\cdot ),$ and $E^3_4(N\cdot )$ (see Lemma \ref{exam-1}). It satisfies the requirement stated in 
Theorem \ref{ithm-3}, and a direct application of Theorem \ref{ithm-3} is stated in Corollary \ref{exam-3}. But 
the following result is an improvement with a similar proof (see Theorem \ref{exam-4}):

\begin{Thm} Let $N\ge 2$. 
Then, there exists infinitely many  pairs $(\alpha, \beta)\in \mathbb Z^2$ such that 
$X_0(N)$ is birational with 
$\cal C(\Delta,  E^3_4,  \alpha \Delta(N\cdot )+ \beta E^3_4(N\cdot ))$,
and 
$$
\deg{\cal C(\Delta,  E^3_4,  \alpha \Delta(N\cdot )+ \beta E^3_4(N\cdot ))}=\psi(N).
$$
\end{Thm}

\vskip .2in 

In closing the introduction we should mention several other works which construct plane models of modular curves 
(\cite{bnmjk}, \cite{sgal}, \cite{mshi}, \cite{yy}). They use strategies which are quite different than ours.

I would like to thank M. Kazalicki and G. Savin for some useful discussion about integral structure on the spaces 
of cusp forms.

\section{Preliminaries}\label{prelim}

In this section we recall from (\cite{Miyake}, 2.3) some notions related to the theory of divisors 
of modular forms and state a preliminary result. In this section $\Gamma$ is any Fuchsian group of the 
first kind.

Let $m\ge 2$ be an even integer and $f\in M_{m}(\Gamma)-\{0\}$. Then
$\nu_{z-\xi}(f)$ the order of the holomorphic function $f$ at $\xi$.
For each $\gamma\in \Gamma$, the functional equation
$f(\gamma.z)=j(\gamma, z)^m f(z)$, $z\in \mathbb H$, shows that 
$\nu_{z-\xi}(f)=\nu_{z-\xi'}(f)$, where $\xi'=\gamma.\xi$. 
Also, if we let
$$
e_{\xi} =\# \left(\Gamma_{\xi}/\Gamma\cap \{\pm 1\}\right),
$$
then $e_{\xi}=e_{\xi'}$. The point $\xi\in \mathbb H$ is elliptic if $e_\xi>1$. Next, following (\cite{Miyake}, 2.3), we define
$$
\nu_\xi(f)=\nu_{z-\xi}(f)/e_{\xi}.
$$
Clearly, $\nu_{\xi}=\nu_{\xi'}$, and we may let 
$$
\nu_{\mathfrak a_\xi}(f)=\nu_\xi(f),
$$ 
where 
$$\text{$\mathfrak a_\xi\in \mathfrak R_\Gamma$ is a projection of $\xi$
to $\mathfrak R_\Gamma$,} 
$$
a notation we use throughout this paper.

If $x\in \bbR\cup \{\infty\}$ is a cusp for $\Gamma$, then we define 
$\nu_x(f)$ as follows. Let $\sigma\in SL_2(\bbR)$ such that
$\sigma.x=\infty$. We write 
$$
\{\pm 1\} \sigma \Gamma_{x}\sigma^{-1}= \{\pm 1\}\left\{\left(\begin{matrix}1 & lh'\\ 0 &
    1\end{matrix}\right); \ \ l\in \bbZ\right\},
$$
where $h'>0$. Then we write the Fourier expansion of $f$ at $x$ as follows:
$$
(f|_m \sigma^{-1})(\sigma.z)= \sum_{n=1}^\infty a_n e^{2\pi
  \sqrt{-1}n\sigma.z/h'}.
$$

We let 
$$
\nu_x(f)=N\ge 0,
$$
where $N$ is defined by $a_0=a_1=\cdots =a_{N-1}=0$, $a_N\neq 0$. One
easily see that this definition does not depend on $\sigma$. Also, 
if $x'=\gamma.x$, then
$\nu_{x}(f)=\nu_{x}(f)$. Hence, if $\mathfrak b_x\in \mathfrak
R_\Gamma$ is a cusp corresponding to $x$, then we may define
$$
\nu_{\mathfrak b_x}=\nu_x(f). 
$$

Put
$$
\text{div}{(f)}=\sum_{\mathfrak a\in \mathfrak
R_\Gamma} \nu_{\mathfrak a}(f) \mathfrak a \in \ \  \bbQ\otimes \text{Div}(\mathfrak R_\Gamma),
$$
where $\text{Div}(\mathfrak R_\Gamma)$ is the group of (integral) divisors on 
$\mathfrak R_\Gamma$.

Using (\cite{Miyake}, 2.3), this sum is finite i.e., $ \nu_{\mathfrak a}(f)\neq 0$
for only a finitely many points. We let 
$$
\text{deg}(\text{div}{(f)})=\sum_{\mathfrak a\in \mathfrak
R_\Gamma} \nu_{\mathfrak a}(f).
$$

Let $\mathfrak d_i\in \bbQ\otimes \text{Div}(\mathfrak R_\Gamma)$, $i=1, 2$. Then we say 
that $\mathfrak d_1\ge \mathfrak d_2$ if their difference $\mathfrak d_1 - \mathfrak d_2$ 
belongs to $\text{Div}(\mathfrak R_\Gamma)$ and is non--negative in the usual sense.

\vskip .2in

\begin{Lem}\label{prelim-1} Assume that $m\ge 4$ is an even integer. Assume that 
$f\in  M_m(\Gamma)$, $f\neq 0$. Let $t$ be the number of inequivalent cusps  for $\Gamma$.  Then we have the following:
\begin{itemize}

\item[(i)] For $\mathfrak a\in \mathfrak
R_\Gamma$, we have  $\nu_{\mathfrak a}(f) \ge 0$.

\item [(ii)]  For a cusp $\mathfrak a\in \mathfrak R_\Gamma$, we have that 
$\nu_{\mathfrak a}(f)\ge 1$ is an integer.

\item[(iii)] If  $\mathfrak a\in \mathfrak
R_\Gamma$ is not an elliptic point or a cusp, then $\nu_{\mathfrak a}(f)\ge 0$
is an integer.  If  $\mathfrak a\in \mathfrak
R_\Gamma$ is an elliptic point, then $\nu_{\mathfrak a}(f)-\frac{m}{2}(1-1/e_{\mathfrak a})$ is 
an integer.

\item[(iv)]Let $g(\Gamma)$ be the genus of $ \mathfrak R_\Gamma$. Then 
 $$
\text{deg}(\text{div}{(f)})=m(g(\Gamma)-1)+ \frac{m}{2}\left(t+ \sum_{\mathfrak a\in \mathfrak
R_\Gamma, \ \ elliptic} (1-1/e_{\mathfrak a})\right).
$$

\item[(v)] Let $[x]$ denote the largest integer $\le x$ for $x\in
  \bbR$.  Then

\begin{align*}
\dim S_m(\Gamma) &=
(m-1)(g(\Gamma)-1)+(\frac{m}{2}-1)t+ \sum_{\substack{\mathfrak a\in \mathfrak
R_\Gamma, \\ elliptic}} \left[\frac{m}{2}(1-1/e_{\mathfrak a})\right]\\
\dim M_m(\Gamma)&=\dim S_m(\Gamma)+t.
\end{align*}

\item[(vi)] There exists an integral divisor $\mathfrak c'_f\ge 0$ of degree 
$\dim M_m(\Gamma)+ g(\Gamma)-1$ such that
\begin{align*}
\text{div}{(f)}= & \mathfrak c'_f+ \sum_{\mathfrak a\in \mathfrak
R_\Gamma, \ \ elliptic} \left(\frac{m}{2}(1-1/e_{\mathfrak a}) -
\left[\frac{m}{2}(1-1/e_{\mathfrak
    a})\right]\right)\mathfrak a.
\end{align*}
\end{itemize}
\end{Lem}
\begin{proof} The claims (i)--(v) are standard (\cite{Miyake}, 2.3, 2.5). The claim (vi) follows from (iii), 
(iv), and (v) (see Lemma 4-1 in \cite{Muic}).
\end{proof}

\vskip .2in

If $f\in S_m(\Gamma)$, we can define an   integral divisor $\mathfrak c'_f\ge 0$ of degree 
$\dim S_m(\Gamma)+g(\Gamma)-1$ by 

\begin{equation}\label{2div}
\mathfrak c_f\overset{def}{=}\mathfrak c'_f-
\sum_{\substack{\mathfrak b \in \mathfrak
R_\Gamma, \\ cusp}} \mathfrak b.
\end{equation}

\vskip .2in 
We end this section with an observation that we use later in the paper. We leave to the reader to check the details.
We introduce some notation (see \cite{Miyake}, 2.1). Let $\alpha\in GL_2^+(\mathbb R)$. Then, the map 
$f\longmapsto f|_k\alpha\overset{def}{=}\det{(\alpha)}^{-k/2}j(\alpha, \cdot)^{-k}f(\alpha. \cdot)$ is an isomorphism 
of vector spaces $M_k(\Gamma)\longrightarrow M_k(\alpha^{-1}\Gamma \alpha)$(resp., 
$S_k(\Gamma)\longrightarrow S_k(\alpha^{-1}\Gamma \alpha)$) (see \cite{Miyake}, (2.1.18)). Also, one easily check that 
$\alpha$ induces the following isomorphism of Riemann surfaces $\mathfrak R_{\alpha^{-1}\Gamma \alpha}\longrightarrow 
\mathfrak R_{\Gamma}$ given in the notation introduced earlier in this section: 
$\mathfrak a'_{z}\overset{def}{=}\left(\alpha^{-1}\Gamma \alpha\right).z \longmapsto \mathfrak a_{\alpha.z}=
\Gamma.(\alpha.z)$, where 
$z\in \mathbb H$, or $z$ is a cusp for $\alpha^{-1}\Gamma \alpha$. Finally, for 
$f\in M_k(\Gamma)$, $f\neq 0$, we have the following: if $\text{div}(f)=\sum_{z}m_z \mathfrak a_z$, then 
$\text{div}(f|_k\alpha)=\sum_{z}m_z \mathfrak a'_{\alpha^{-1}.z}$.

\section{A proof of the formula for the degree}\label{p2}
\vskip .2in

In this section $\Gamma$ is an arbitrary Fuchsian group of the first kind.  
We introduce the objects of our study in the following lemma:

\begin{Lem}\label{p2-1}  Assume that  $m\ge 4$ is an even integer such that 
$\dim M_m(\Gamma)\ge 3$. Let $f, g, h\in M_m(\Gamma)$ be linearly independent.
Then, the image of the map $\varphi: \ \mathfrak R_\Gamma\rightarrow \Bbb P^{2}$
given by
$$
\mathfrak a_z\longmapsto (f(z):  g(z): h(z))
$$
is an irreducible projective curve which we denote by 
$\cal C(f, g, h)$. The degree of $\cal C(f, g, h)$ (i.e., the degree of $P$) is 
$\le \dim M_m(\Gamma)+ g(\Gamma)-1$. Moreover, if $f$, $g$, and $h$ are selected to be cusp forms, then the 
degree is $\le \dim S_m(\Gamma)+ g(\Gamma)-1$.
\end{Lem}
\begin{proof}  Note that $g/f$ and $h/h$ are rational functions on $\mathfrak R_\Gamma$ considered 
as a smooth irreducible projective curve over $\mathbb C$. Thus, the meromorphic map 
$\mathfrak a_z\longmapsto (f(z):  g(z): h(z))$ is actually a rational map 
$$
\mathfrak a_z\longmapsto (1:  g(z)/f(z): h(z)/g(z)).
$$
Hence,  it is regular since $\mathfrak R_\Gamma$ is smooth. The image of the map is clearly not constant. Hence, 
it is an irreducible curve in $\mathbb P^2$.

Let $l$ be the line in $\Bbb P^{2}$ in general position with respect 
to $\cal C(f, g,  h)$. Then, it intersects $\cal C(f, g,h)$ in different points a number of 
which 
is the degree of $\cal C(f, g,h)$. We can  change the coordinate system so that the line $l$ is 
$x_0=0$. 
In new coordinate system,  the map $\mathfrak a_z\mapsto \left(f(z): g(z) : h(z)\right)$ is of 
the form 
$$
\mathfrak a_z\mapsto \left(F(z): G(z) : H(z)\right),
$$
where $F, G, H$ are again linearly independent.
In particular, $F, G, H \neq 0$. 

We write this map  in the form
$$
\mathfrak a_z\mapsto \left(1: G(z)/F(z) : H(z)/F(z)\right).
$$
By Lemma \ref{prelim-1} (vi), we can write
\begin{align*}
&\text{div}{(G/F)}= \text{div}{(G)} - \text{div}{(F)}=
\mathfrak c'_{G}-\mathfrak c'_{F},\\
&\text{div}{(H/F)}= \text{div}{(H)} - \text{div}{(F)}=
\mathfrak c'_{H}-\mathfrak c'_{F}.\\
\end{align*}
We remark that the divisors $\mathfrak c'_{F}, \mathfrak c'_{G}$, and $\mathfrak c'_{H}$ are 
integral divisors of degree $\dim M_m(\Gamma)+ g(\Gamma)-1$ (see Lemma \ref{prelim-1} (vi)).

Now, we intersect $\cal C(f, g, h)$ with the line $x_0=0$. The intersection points of 
intersection are contained among the points in the support of 
$\mathfrak c'_{F}$. The claim about the degree  follows since the support can not have more than  
$\dim M_m(\Gamma)+ g(\Gamma)-1$ points. If we deal with the cusp forms, then we can slightly improve the last argument 
using (\ref{2div}). This proves the last claim of the lemma.
\end{proof}

\vskip .2in 
Next, we define the degree of the covering map $\varphi: \mathfrak R_\Gamma\longrightarrow 
\cal C(f, g, h)$ (see Lemma \ref{p2-3} below). This would be a standard fact (see page 31 of \cite{Miyake} for the summary)
if $\cal C(f, g, h)$ would have no singularities. We follow and modify the standard way of defining 
the degree of the map as explained in \cite{Miranda}. 
\vskip .2in

First, we observe  

\begin{Lem}\label{p2-2} Maintaing the assumptions of Lemma \ref{p2-1}, the preimage $\varphi^{-1}(q)$ 
is finite for any $q\in \cal C(f, g, h)$.
\end{Lem}
\begin{proof}Let Let $q=(x_0:x_1:x_2)$.  Without loss of generality we may assume that 
$x_2\neq 0$. Then, $h$ is not identically zero.  Since $f$, $g$, and $h$ are linearly independent, 
the quotient $f/h$ is not constant. 
Now, by the standard theory of compact Riemann surfaces 
(see for example summary on pages 30--31 in \cite{Miyake}), the regular map $\mathfrak R_\Gamma\longrightarrow 
\mathbb P^1$ defined by $f/h$  has finite preimages. 
\end{proof}

\vskip .2in
We let $V$ be the complement of finitely many  points in 
$\cal C(f, g, h)$ where this curve is singular. We let $U$ be the preimage 
of $V$ in $\mathfrak R_\Gamma$. By Lemma \ref{p2-2}, it is a complement of finitely many points which maps to the set 
of singular points in  $\cal C(f, g, h)$.  Thus, both $U$ and $V$ are open Riemann surfaces and we have a holomorphic  
surjective map $\varphi|_U: U\rightarrow V$.

The  multiplicity $mult_p(\varphi|_U)$ of $p\in U$ of $\varphi|_U$ is defined in the usual way:   using suitable 
local coordinates, in a neighborhood of $p$, the map   $\varphi|_U$  is of the form $w\mapsto w^{mult_p(\varphi|_U)}$ 
($w=0$ corresponds to $p$). As usual, following  \cite{Miranda}, we let

\begin{equation}\label{degree}
\deg_{q}(\varphi|_U)=\sum_{p\in \varphi^{-1}(q)} mult_p(\varphi|_U), \ \ q\in U.
\end{equation}
The following lemma is a variant 
of the standard argument (i.e., the case when  $\cal C(f, g, h)$ has no singularities)

\begin{Lem}\label{p2-3}
The map $q\mapsto \deg_{q}(\varphi|_U)$ is constant on $U$. In this way we define a degree of 
of $\varphi$ (since $U$ is uniquely determined by $\varphi$), and  we denoted
by $\deg(\varphi)$.
\end{Lem}
\begin{proof} First,  $\cal C(f, g, h)$ is connected since it is a continuous 
image of the connected set $\mathfrak R_\Gamma$.   Then, since $U$ is a complement 
of finitely many points in $\cal C(f, g, h)$, $U$ is connected. Thus, it is enough to show that 
$q\mapsto \deg_{q}(\varphi|_U)$ is locally constant.

Let us show that $q\mapsto \deg_{q}(\varphi|U)$ is locally constant. Let us fix $q \in U$. For each of finitely many points 
$p\in \varphi^{-1}(q)$, we select a neighborhood (charts) $U_p$ of $p$, and a neighborhood $W$ of 
$q$ such that $\varphi(U_p)\subset W$, and   $\varphi$  is of the form
$w_p\mapsto w_p^{mult_p(\varphi|U)}$ ($w_p=0$ corresponds to $p$). By shrinking $U_p$, 
we may assume that they are all disjoint.  Then, for each  $q'\in W-\{q\}$, there are 
$mult_p(\varphi|U)$ different points from $U_p$ which maps to $q'$. Thus, there are 
$\deg_{q}(\varphi|U)=\sum_{p\in \varphi^{-1}(q)} mult_p(\varphi|_U)$ 
different points from the union $\cup_{p\in \varphi^{-1}(q)} U_p$ which maps 
to $q'$. 

We may think that the chart $W$ is given by an open circle $|u|< r$, $u=0$ corresponds to $q$, and 
we may define neighborhoods $W_\rho$, $\rho< r$, of $\rho$ which corresponds to  open circles $|u|< \rho$.
We show that for suitable small $\rho$, none of the preimages is left out i.e., for each 
 $q'\in W_\rho-\{q\}$ we have $\varphi^{-1}(q)\subset \cup_{p\in \varphi^{-1}(q)} U_p$.  If not, 
then  there is a  sequence of points 
such that $q_n\rightarrow q$, and there  is a sequence of points $p_n\in U -\cup_{p\in 
\varphi^{-1}(q)} U_p$ such that $\varphi(p_n)=q_n$. 

The key point is the fact that the sequence $p_n$ belongs to the complement of 
$\cup_{p\in \varphi^{-1}(q)}  U_p$ in $\mathfrak R_\Gamma$. But $\mathfrak R_\Gamma$ is compact so 
the sequence  $p_n$ has a convergent subsequence; we may assume that $p_n$ itself is a
convergent. Let $\lim_n p_n=p'$. Then,  $\varphi(p')=q$ and 
$p'\not \in \cup_{p\in \varphi^{-1}(q)}  U_p$. This is clearly a contradiction.

Thus, for suitable small $\rho$, preimages of $q'\in W_\rho-\{q\}$ 
consist of exactly $\deg_{q}(\varphi|U)$ different points. Clearly, in a neighborhood  of them 
the map is of the form $w\mapsto  w$, which implies 
$$
\deg_{q'}(\varphi|U)=\deg_{q}(\varphi|U),
$$ 
or otherwise there would exist a point $q{'}\in W_\rho-\{q\}$,  and a
preimage $p'$ of $q'$ in some $U_p$ such that the map is in the neighborhood of $p'$ is of the form 
$w\mapsto w^{l}$, $l\ge 1$. Then, for $q{''}\in W_\rho-\{q\}$ near $q'$, in the neighborhood of $p'$ 
the point $q{''}$ would have at least two preimages.  This would result in 
at least $mult_p(\varphi|_U)+1$ of preimages of $q{''}$ in $U_p$ which is impossible. This proves the lemma
\end{proof}

\vskip .2in

Now, we state and prove the main result of the present section.

\begin{Thm}\label{p2-120}  Assume that  $m\ge 4$ is an even integer such that 
$\dim M_m(\Gamma)\ge 3$. Let $f, g, h\in M_m(\Gamma)$ be linearly independent.
Then, we have the following:
$$
\deg{(\varphi)}\cdot \deg{C(f, g, h)}=\dim M_m(\Gamma)+ g(\Gamma)-1 - \sum_{\mathfrak a\in \mathfrak R_\Gamma} 
\min{\left(\mathfrak c'_{f}(\mathfrak a),  \mathfrak c'_{g}(\mathfrak a), 
\mathfrak c'_{h}(\mathfrak a) \right)}.
$$
Moreover, if $f, g, h\in S_m(\Gamma)$, then 
$$
\deg{(\varphi)}\cdot \deg{C(f, g, h)}=\dim S_m(\Gamma)+ g(\Gamma)-1 - \sum_{\mathfrak a\in \mathfrak R_\Gamma} 
\min{\left(\mathfrak c_{f}(\mathfrak a),  \mathfrak c_{g}(\mathfrak a), 
\mathfrak c_{h}(\mathfrak a) \right)}.
$$
In particular, if $\varphi$ is birational, then $\deg{(\varphi)}=1$ (since it is generically injective), and we have 
in either case a simple formula for the  degree of $\deg{C(f, g, h)}$ in terms of $f, g$, and $h$.
\end{Thm}
\begin{proof} The formula in the case $f, g, h\in S_m(\Gamma)$ follows at once from the general case using 
(\ref{2div}).  We consider the general case $f, g, h\in M_m(\Gamma)$.

In the first step of the proof, we associate the linear system to the map $\varphi$.
Assume that  $k\in M_m(\Gamma)$ is non--zero. We write the map in the form
$$
\varphi(\mathfrak a_z)=(f(z):  g(z): h(z))=(F(z):G(z):H(z)),
$$
where $F=f/k$, $G=g/k$, and $H=h/k$.  Then, we define the divisor $\mathfrak d_k$ in the usual way \cite{Miranda} using
$$
\mathfrak d_k(\mathfrak a)=-\min{\left(\text{div}(F)(\mathfrak a), \text{div}(G)(\mathfrak a), 
 \text{div}(H)(\mathfrak a) 
\right)}, \ \ \mathfrak a\in \mathfrak R_\Gamma.
$$

Using Lemma \ref{prelim-1} (vi), we compute 
\begin{equation}\label{p2-40}
\begin{aligned}
{\mathfrak d_k} & = -\sum_{\mathfrak a\in \mathfrak R_\Gamma} 
\min{\left(\text{div}(F)(\mathfrak a), \text{div}(G)(\mathfrak a),  \text{div}(H)(\mathfrak a) 
\right)}\mathfrak a\\
& = -\sum_{\mathfrak a\in \mathfrak R_\Gamma} 
\min{\left(\mathfrak c'_{f}(\mathfrak a)- \mathfrak c'_{k}(\mathfrak a),
\mathfrak c'_{g}(\mathfrak a)- \mathfrak c'_{k}(\mathfrak a), 
\mathfrak c'_{h}(\mathfrak a)- \mathfrak c'_{k}(\mathfrak a)\right)}\mathfrak a\\
& =\mathfrak c'_{k} - \sum_{\mathfrak a\in \mathfrak R_\Gamma} 
\min{\left(\mathfrak c'_{f}(\mathfrak a),  \mathfrak c'_{g}(\mathfrak a), 
\mathfrak c'_{h}(\mathfrak a) \right)}\mathfrak a.
\end{aligned}
\end{equation}

The computation in (\ref{p2-40}) shows that different $k$'s determine the same linear system $|\mathfrak d_k|$. 
We shall select $k=h$ in the sequel, and let 
\begin{equation}\label{p2-41}
\mathfrak d=\mathfrak d_h.
\end{equation}

\vskip .2in 

Let  $l\subset \Bbb P^2$ be a line. Let us write 
$\varphi^*(l)$ be the hyperplane divisor of the map $\varphi: \mathfrak R_\Gamma 
\rightarrow \cal C(f, g, h)$. By definition, if we write the equation of the line 
$l$ as follows: $a_0x_0+a_1x_1+a_2x_2=0$, then 
\begin{equation}\label{p2-42}
\varphi^*(l)=\text{div}(a_0 F+ a_1 G+ a_2)+\mathfrak d.
\end{equation}

\vskip .2in
If $\cal C(f, g, h)$ is smooth (i.e., an embedded Riemann surface), then one would 
define the
intersection divisor $\text{div}(l)$ of a line $l$ as follows. Let $q\in l\cap \cal 
C(f, g, h)$.
We select a coordinate function $x_i$ which does not vanish at $q$ and let 
$\text{div}(l)(q)=\ord_q (l/x_i)$. This is idependent of the coordinate function used. 
We have
$$
\text{div}(l)=\sum_{q\in l\cap \cal C(f, g, h)} \text{div}(l)(q) q.
$$
Still assuming that $\cal C(f, g, h)$ is smooth, we have
\begin{equation}\label{p2-4}
\deg{(\text{div}(l))}=\deg{C(f, g, h)},
\end{equation}
and 
\begin{equation}\label{p2-5}
\deg{(\mathfrak d)}=\deg{(\varphi^*(l))}=\deg{(\varphi)}\cdot \deg{(\text{div}(l))}.
\end{equation}

\vskip .2in

Of course, in general $C(f, g, h)$ can have singularities outside $V$ (introduced before the statement of 
Lemma \ref{p2-3}). So,  we modify the proof of above formulas such that  they hold in generality we need. First of all, 
we restrict ourselves to the lines  such that  
$l\cap \cal C(f, g, h)\subset V$. Then, we may define $\text{div}(l)$ as before. 

\vskip .2in

\begin{Lem}\label{p2-6} Let $l\subset \Bbb P^2$ be any line such that 
$l\cap \cal C(f, g, h)\subset V$. Then, (\ref{p2-4}) and (\ref{p2-5}) hold.
Moreover, we can select a line $l\subset \Bbb P^2$ such that 
$l\cap \cal C(f, g, h)\subset V$ and $l\cap \cal C(f, g, h)$ consists of 
$\deg{C(f, g, h)}$ different points.
\end{Lem}
\begin{proof} With the aid of Lemma \ref{p2-3}, we easily 
adapt the proof of  (\cite{Miranda}, Proposition 4.23) to prove (\ref{p2-5}). 
We leave details to the reader.

To show  (\ref{p2-4}), we adapt the classical argument with resultants. 
We may assume $(0:0:1)\not \in \cal C(f, g, h)$. 
We look at the family of lines $l_\lambda$ given by $x_0-\lambda x_1=0$ that pass 
through this point. Since there are just finitely many points in
$\cal C(f, g, h)-V$, for all but finitely many $\lambda$'s we have the following: 
$l_\lambda\cap \cal C(f, g, h)\subset V$.

We observe that $x_0-\lambda x_1$ and $x_1$ never vanish simultaneously on 
$\cal C(f, g, h)$ since  $(0:0:1)\not \in \cal C(f, g, h)$. Thus, 
$\text{div}(x_0-\lambda x_1)$ is determined by  $x_0/x_1-\lambda$ at any point of 
intersection of $l_\lambda$ and $\cal C(f, g, h)$. 

The intersection of  $x_0-\lambda x_1$ with $\cal C(f, g, h)$  is of the form 
$(\lambda: 1: \mu)$. If we let $P$ be the irreducible homogeneous polynomial which 
locus is $\cal C(f, g, h)$,  then the equation for $\mu$ is given 
by $P(\lambda, 1, \mu)=0$. Since $(0:0:1)\not \in \cal C(f, g, h)$, we may write
(up to a non--zero constant depending on $P$ only)
$$
P(\lambda, 1, \mu)=\mu^{\deg{C(f, g, h)}}+\sum_{i=0}^{\deg{C(f, g, h)}-1} a_i(\lambda)\mu^i,
$$
where $a_i$ is polynomial in $\lambda$. The discriminant of $P$ with respect to $\mu$ 
(that is, a resultant of $P(\lambda, 1, \mu)$ and $\frac{\partial}{\partial \mu}
P(\lambda, 1, \mu)$ ) is a polynomial of $\lambda$ which does not vanish identically.\footnote
{Otherwise,  if $X_0, X_1, X_2$ denote independent variables, then the resultant of 
$P(X_0, X_1, X_2)$ and $\frac{\partial }{\partial X_2}P(X_0, X_1, X_2)$  which is a homogeneous
polynomial in $X_0, X_1$  must be zero. But then $P(X_0, X_1, X_2)$ and $\frac{\partial }{\partial X_2}P(X_0, X_1, X_2)$ 
would have a common irreducible factor. This factor is obviously 
$P(X_0, X_1, X_2)$ since it is  irreducible and of higher 
degree than its derivative. This a contradiction since this polynomial has the degree $>$ than its derivative.}

 At all but finitely many points $\lambda$, we have 
that the equation for $\mu$ 
\begin{equation}\label{p2-7}
P(\lambda, 1, \mu)=\mu^{\deg{C(f, g, h)}}+\sum_{i=0}^{\deg{C(f, g, h)}-1} a_i(\lambda)\mu^i=0,
\end{equation}
satisfies 
$$
P(\lambda, 1, \mu)=0\implies \frac{\partial }{\partial X_2}P(\lambda, 1, \mu)\neq 0.
$$
This means that for such $\lambda$, we have $\deg{C(f, g, h)}$ different solutions for 
$\mu$. By the Implicit function theorem, $\frac{\partial }{\partial X_2}P(\lambda, 1, \mu)\neq 0$
means that near the point $(\lambda: 1: \mu)$, the local coordinate 
is $x_0/x_1$. Hence,
$$
\ord_{(\lambda: 1: \mu)}(x_0-\lambda x_1)=1.
$$
Finally, for $\lambda$ which makes the discriminant non--vanishing, we have
$$
\text{div}(x_0-\lambda x_1)=\sum_{(\lambda: 1: \mu)} (\lambda: 1: \mu),
$$
where $\mu$ runs over all solutions of (\ref{p2-7}). This implies 
$$
\deg{(\text{div}(x_0-\lambda x_1))}=\deg{C(f, g, h)}.
$$
This proves the claim about the degree.
\end{proof}

\vskip .2in 
Having completed the proof of Lemma \ref{p2-6}, the proof of Theorem \ref{p2-120} is easy to complete.
Let $l\subset \mathbb P^2$ be any line. Then, by Lemma \ref{p2-6},  (\ref{p2-4}) and (\ref{p2-5}) hold. 
So, if we combine them with (\ref{p2-40}) (with $k=h$), we obtain

\begin{align*}
\deg{(\varphi)}\cdot \deg{C(f, g, h)}&= \deg{\left(\mathfrak c'_{h} - \sum_{\mathfrak a\in \mathfrak R_\Gamma} 
\min{\left(\mathfrak c'_{f}(\mathfrak a),  \mathfrak c'_{g}(\mathfrak a), 
\mathfrak c'_{h}(\mathfrak a) \right)}\mathfrak a\right)}\\
&=
\deg{(\mathfrak c'_{h})}-  \sum_{\mathfrak a\in \mathfrak R_\Gamma} 
\min{\left(\mathfrak c'_{f}(\mathfrak a),  \mathfrak c'_{g}(\mathfrak a), 
\mathfrak c'_{h}(\mathfrak a) \right)}\\
&=\dim M_m(\Gamma)+ g(\Gamma)-1 - \sum_{\mathfrak a\in \mathfrak R_\Gamma} 
\min{\left(\mathfrak c'_{f}(\mathfrak a),  \mathfrak c'_{g}(\mathfrak a), 
\mathfrak c'_{h}(\mathfrak a) \right)}.
\end{align*}
The last equality follows by Lemma \ref{prelim-1} (vi).
\end{proof}

\section{A generic construction of birational maps}\label{gen}

In this section, we let $t_m=\dim S_m(\Gamma)$. The goal of this section is to construct 
various models of the curve $\mathfrak R_\Gamma$, where $\Gamma$ is any Fuchsian group of the 
first kind. 

\vskip .2in

\begin{Lem}\label{ls-thm}  Let $m\ge 4$ be an even integer such that $t_m\ge g(\Gamma)+2$. 
Then,  the field of rational functions
$\Bbb C (\mathfrak R_\Gamma)$ is generated over $\Bbb  C$ by the rational functions
$f_i/f_0$, $1\le i \le t_m-1$, where $f_0, \ldots, f_{t_m-1}$ is a basis of $S_m(\Gamma)$.
\end{Lem}
\begin{proof} This is (\cite{Muic1}, Corollary  3-7). Let us sketch the proof. For $f\in S_m(\Gamma)$, we 
consider $L(\mathfrak c_f)$ which is by definition the space of all $F\in \Bbb C (\mathfrak R_\Gamma)$ such that
$\text{div}(F)+\mathfrak c_f\ge 0$. In (\cite{Muic1}, Proposition 2-10) we show that 
$L(\mathfrak c_f)=\left\{g/f; \ \ g\in S_m(\Gamma)\right\}$ assuming only that $m\ge 4$ and $t_m\ge 1$.
The proof of this is similar to the proof of (\cite{Muic}, Theorem 4-15) using computations on pages $17$ and $18$ of
\cite{Muic}. Then, we construct the embedding $\mathfrak R_\Gamma\longrightarrow \bbP^{t_m-1}$ using the holomorphic map
$$
\mathfrak a_z\mapsto \left(f_0(z): \cdots : f_{t_m-1}(z)\right)= \left(f_0(z)/f(z): \cdots : f_{t_m-1}(z)/f(z)\right).
$$
A computation similar to that in (\ref{p2-40}) shows that
this map is attached to the linear system $|\mathfrak c_f|$ as we demonstrate this in (\cite{Muic1}, Theorem 3-3).
If $t_m\ge g(\Gamma)+2$, then $\mathfrak c_f$ is very ample. So, the map is an embedding.  The claim of the lemma 
is an obvious consequence of this.
\end{proof}

\vskip .2in

\begin{Lem}\label{prelim-8} Let $\xi\in X$ or let $\xi$ be a cusp for $\Gamma$. Let
$m\ge 4$ be an even integer such that $t_m\ge g(\Gamma)+1$. Then,  
there exists $f\in S_m(\Gamma)$ such that $\mathfrak c_{f}(\mathfrak a_\xi)= 0$.
\end{Lem}
\begin{proof} This is (\cite{Muic1}, Lemma 2-9). The proof of this lemma is a straightforward generalization of
computations made in (\cite{Muic}, Section 4, pages $17$ and $18$). 
\end{proof}

\vskip .2in

\begin{Lem}\label{prelim-8-cor} Assume that  $m\ge 4$ is an even integer. 
Let $W\not\subset S_m(\Gamma)$ be a subspace which separates the points of $\mathfrak R_\Gamma$.
Select a basis $f_0, \ldots, f_{s-1}$ for $W$. Then, for each $\xi\in X$ or a cusp for $\Gamma$, 
there exists $i$ such that $\mathfrak c'_{f_i}(\mathfrak a_\xi)= 0$.
\end{Lem}
\begin{proof} Let $f\in W$, $f\neq 0$, be an arbitrary form. Consider the linear space 
$$
L(\mathfrak c'_f)=\{F\in \Bbb C(\mathfrak R_\Gamma); \ \ \text{div}(F)+\mathfrak c'_f\ge 0\}.
$$
Then, it contains a linear subspace $W_1$ consisting of all quotients $g/f$, $g\in W$. This is so, 
since,  for $g\neq 0$, by Lemma \ref{prelim-1} (vi) we obtain
$$
\text{div}\left(\frac{g}{f}\right) + \mathfrak c'_f=\text{div}(g)-\text{div}(f)+ \mathfrak c'_f=
\mathfrak c'_g-\mathfrak c'_f+\mathfrak c'_f=\mathfrak c'_g\ge 0.
$$
Assume that the claim of the lemma is not true, then for all $i$ we have $\mathfrak c'_{f_i}(\mathfrak a_\xi)\ge 1$. 
So, we have 
$$
\text{div}\left(\frac{f_i}{f}\right) + \mathfrak c'_f-\mathfrak a_\xi=\text{div}(f_i)-\text{div}(f)+ \mathfrak c'_f=
\mathfrak c'_{f_i}-\mathfrak c'_f+\mathfrak c'_f -\mathfrak a_\xi =\mathfrak c'_{f_i}-\mathfrak a_\xi\ge 0.
$$
Thus, $f_i/f\in L(\mathfrak c'_f -\mathfrak a_\xi)$. This implies that $W_1\subset  L(\mathfrak c'_f -\mathfrak a_\xi)$.
But there exists $g\in W$ such that $\mathfrak c'_{g}(\mathfrak a_\xi)= 0$.
Then $\mathfrak c'_{g}-\mathfrak a_\xi\ge 0$ is clearly not true.
\end{proof}

\vskip .2in 

\begin{Lem}\label{prelim-8-cor-1} Assume that  $m\ge 4$ is an even integer. 
Let $W\subset S_m(\Gamma)$ be a subspace which separates the points of $\mathfrak R_\Gamma$.
Select a basis $f_0, \ldots, f_{s-1}$ for $W$. Then, for each $\xi\in X$ or a cusp for $\Gamma$, 
there exists $i$ such that $\mathfrak c_{f_i}(\mathfrak a_\xi)= 0$.
\end{Lem}
\begin{proof} In view of (\ref{2div}) this has the same proof as the previous lemma. 
\end{proof}
\vskip .2in 

\begin{Lem}\label{p2-8}  
Assume that  $m\ge 4$ is an even integer. Let $W\subset M_m(\Gamma)$, $\dim W\ge 3$,  be a subspace which generates
 the field of rational 
functions $\mathbb C(\mathfrak R_\Gamma)$, and separates the points of $\mathfrak R_\Gamma$.
Then there exists  a non--empty Zariski open set $\cal U\subset W$ such that for any 
$h\in \cal U$, we have that the field of rational functions
$\Bbb C (\mathfrak R_\Gamma)$ is generated over $\Bbb  C$ by the rational functions
$g/f$ and $h/f$, and  $\supp{(\mathfrak c_{f})} \cap \supp{(\mathfrak c_{h})}=\emptyset $ if 
$W\subset S_m(\Gamma)$ or $\supp{(\mathfrak c'_{f})} \cap \supp{(\mathfrak c'_{h})}=\emptyset $ if 
$W\not\subset S_m(\Gamma)$.
\end{Lem}
\begin{proof} For the matter of notation, we consider the case $W\subset S_m(\Gamma)$. In the other case, one needs 
to replace all $\mathfrak c$ with  $\mathfrak c'$.

We select a basis $f_0, \ldots, f_{s-1}$ of $W$, $\dim W=s\ge 3$ such that 
$f=f_0$ and $g=f_1$. By the assumption on $W$, the field of rational functions
$\Bbb C (\mathfrak R_\Gamma)$ is generated over $\Bbb  C$ by all $f_i/f_0$, $1\le i\le s$. 
We let 
$$
K=\Bbb C(f_1/f_0),
$$
and 
$$
L= \Bbb C (\mathfrak R_\Gamma)=
\Bbb C(f_1/f_0, \ldots, f_{s-1}/f_0)=K(f_2/f_0, \ldots, f_{s-1}/f_0).
$$
By Lemma \ref{p2-1},
$f_2/f_0, \ldots, f_{s-1}/f_0$ are all algebraic over $K$. Thus, the field $L$ 
is a finite algebraic extension of $K$. It is also obviously separable. Hence, 
by a variant of a proof of Primitive Element 
Theorem there exists $\lambda_2, \ldots, \lambda_{s-1}\in \Bbb C$ such that  
$$
L= K((\lambda_2 f_2+\cdots + \lambda_{s-1} f_{s-1})/f_0)=
\Bbb C(f_1/f_0, (\lambda_2 f_2+\cdots + \lambda_{s-1} f_{s-1})/f_0).
$$

Now, we explain the systematic way to get them all.  For $(\lambda_2, \ldots, \lambda_{s-1})\in \Bbb C^{s-2}$, 
we consider  the characteristic polynomial
$$
P(X, \lambda_2, \ldots, \lambda_{s-1})=
 \det{\left(X\cdot Id_L- T_{(\lambda_2 f_2+\cdots + \lambda_{s-1} f_{s-1})/f_0}\right)},
$$
where $T_x: L\rightarrow L$, is an $K$--endomorphism given by $T_x(y)=xy$, and 
$Id_L$ is identity on $L$. The resultant $R$ with respect to the variable $X$ 
of the polynomial $P(X, \lambda_2, \ldots, \lambda_{s-1})$ and its 
derivative $\frac{\partial}{\partial X}P(X, \lambda_2, \ldots, \lambda_{s-1})$
is a polynomial in $\lambda_2, \ldots, \lambda_{s-1}$. 

If 
$R(\lambda_2, \ldots, \lambda_{s-1})\neq 0$, then 
$(\lambda_2 f_2+\cdots + \lambda_{s-1} f_{s-1})/f_0$ generate $L$ over $K$. 
Indeed, the characteristic polynomial $P(X, \lambda_2, \ldots, \lambda_{s-1})$ has no
multiple roots in the algebraic closure of $L$. It also has the same roots as the 
minimal polynomial of $(\lambda_2 f_2+\cdots + \lambda_{s-1} f_{s-1})/f_0$. Thus, 
they are equal. Since the degree of the characteristic polynomial is equal to $[L:K]$, 
this element must be primitive. The first part of the proof assures that 
the resultant is not identically zero so that these considerations make sense. 

Hence, primitive elements for the extension $K\subset L$ are constructed from the set of all 
$$
h=\lambda_2 f_2+\cdots + \lambda_{s-1} f_{s-1} \in \mathbb C f_2\oplus \cdots \oplus  \mathbb C f_{s-1}
$$
which belong to  the Zariski open set defined by 
\begin{equation}\label{p2-50}
R(\lambda_2, \ldots, \lambda_{s-1})\neq 0.
\end{equation}
It does not affect the thing if we enlarge $h$ to be 
$$
h=\lambda_0f_0+ \lambda_1f_1+\lambda_2 f_2+\cdots + \lambda_{s-1} f_{s-1},
$$
where $\lambda_0, \lambda_1$ are abitrary complex numbers.  This means that $h$ can be selected
from the  Zariski open subset of  $W$ given by (\ref{p2-50}), where we consider 
the resultant $R$ as a polynomial of all variables $\lambda_0, \ldots, \lambda_{s-1}$ but which does not depend 
on the first two variables.

Now, we prove the last part of the lemma.  By the second assumption on $W$ and Lemma \ref{prelim-8-cor-1}, for each  
$\mathfrak a \in \supp{(\mathfrak c_{f_0})}$ there exists 
$i_{\mathfrak a}\in \{1, \ldots, s-1\}$
such that $\mathfrak a \not\in \supp{(\mathfrak c_{f_{i_{\mathfrak a}}})}$. Then, the rational functions 
$f_i/f_{i_{\mathfrak a}}$ are defined at $\mathfrak a$ since we have 
the following  (see Lemma \ref{prelim-1} (vi))
$$
\text{div}\left(\frac{f_i}{f_{i_{\mathfrak a}}}\right)=\text{div}(f_i)-\text{div}(f_{i_{\mathfrak a}})=
\mathfrak c_{f_i}-\mathfrak c_{f_{i_{\mathfrak a}}},
$$
where the right--most difference consists of effective divisors,
so that the point 
$\mathfrak a$ does not belong to the divisors of poles because of $\mathfrak a \not\in 
\supp{(\mathfrak c_{f_{i_{\mathfrak a}}})}$.

Now, we can form the 
following product of non--zero linear forms in $(\lambda_0, \ldots, \lambda_{s-1})\in \Bbb C^{s}$
$$
\prod_{\mathfrak a \in \supp{(\mathfrak c_{f_0})}} \left(\lambda_0 
\frac{f_0}{f_{i_{\mathfrak a}}}(\mathfrak a) +
\lambda_1 \frac{f_1}{f_{i_{\mathfrak a}}}(\mathfrak a) +\cdots + \lambda_{s-1} 
\frac{f_{s-1}}{f_{i_{\mathfrak a}}}(\mathfrak a) 
\right).
$$

For $\sum_{i=0}^{s-1}\lambda_i f_i$ in a Zariski open subset of $W$,  defined by making this product not equal to zero,
neither of $\mathfrak a \in \supp{(\mathfrak c_{f_0})}$ belong to the divisor of zeroes 
$\text{div}_0\left((\sum_{i=0}^{s-1} \lambda_i f_i)/f_{i_{\mathfrak a}}\right)$ of the corresponding rational function.
Since $\mathfrak a \not\in \supp{(\mathfrak c_{f_{i_{\mathfrak a}}})}$ and 
$$
\text{div}_0\left(\frac{\sum_{i=0}^{s-1}\lambda_i f_i}{f_{i_{\mathfrak a}}}\right)
-\text{div}_\infty\left(\frac{\sum_{i=0}^{s-1}\lambda_i f_i}{f_{i_{\mathfrak a}}}\right)=
\text{div}\left(\frac{\sum_{i=0}^{s-1}\lambda_i f_i}{f_{i_{\mathfrak a}}}\right)=
\mathfrak c_{\sum_{i=0}^{s-1}\lambda_i f_i}-\mathfrak c_{f_{i_{\mathfrak a}}},
$$
where the right most expression is a difference of two effective divisors, 
we get 
$$
\mathfrak a \in \supp{(\mathfrak c_{f_0})} \implies 
\mathfrak a \not\in \supp{(\mathfrak c_{\lambda_0  f_0+\lambda_1 f_1+\cdots + 
\lambda_{s-1} f_{s-1}})}.
$$
Combining this with (\ref{p2-50}), we complete the proof of the lemma.
\end{proof}

\vskip .2in 
\begin{Thm}\label{p2-thm}
Assume that  $m\ge 4$ is an even integer. Let $W\subset M_m(\Gamma)$, $\dim W\ge 3$,  be a subspace which generates
 the field of rational 
functions $\mathbb C(\mathfrak R_\Gamma)$, and separates the points of $\mathfrak R_\Gamma$. For example, if 
$\dim S_m(\Gamma) \ge \max{(g(\Gamma)+2, 3)}$, then we can take $W=S_m(\Gamma)$. Let $f, g \in W$ be linearly independent.
Then there exists  a non--empty Zariski open set $\cal U\subset W$ such that for any 
$h\in \cal U$ we have the following:
\begin{itemize}
\item[(i)] $\mathfrak R_\Gamma$ is birationally equivalent to  $\cal C(f, g, h)$, and 
\item[(ii)] $\cal C(f, g, h)$ has degree equal to $\dim M_m(\Gamma)+g(\Gamma)-1$ (resp., 
$\dim S_m(\Gamma)+g(\Gamma)-1$) if $W\not\subset S_m(\Gamma)$ (resp., $W\subset S_m(\Gamma)$).
\end{itemize}
\end{Thm}
\begin{proof} First of all, Lemmas \ref{ls-thm} and \ref{prelim-8} assure that $W=S_m(\Gamma)$
 generates the field of rational 
functions $\mathbb C(\mathfrak R_\Gamma)$, and separates the points of $\mathfrak R_\Gamma$ whenever 
$\dim S_m(\Gamma) \ge \max{(g(\Gamma)+2, 3)}$.

Now, go back to the general subspace which satisfies these conditions. 
We select the set $\cal U\subset W$ given by Lemma \ref{p2-8}.
Since, by Lemma \ref{p2-8}, the field of rational functions $\Bbb C (\mathfrak R_\Gamma)$ is generated over 
$\Bbb  C$ by the rational functions $g/f$ and $h/f$, we immediatelly get that the map given by Lemma \ref{p2-1}
is birational equivalence. This proves (i). 

Now, we prove (ii). Since, $\varphi$ is birational, by Theorem \ref{p2-120},  we get 
$\deg{(\varphi)}=1$, and, if $W\not\subset S_m(\Gamma)$, then  

\begin{align*}
\deg{C(f, g, h)}=\deg{(\varphi)}\cdot \deg{C(f, g, h)}& =\dim M_m(\Gamma)+g(\Gamma)-1 - \sum_{\mathfrak a\in \mathfrak R_\Gamma} 
\min{\left(\mathfrak c'_{f}(\mathfrak a),  \mathfrak c'_{g}(\mathfrak a), 
\mathfrak c'_{h}(\mathfrak a) \right)}\\
&=\dim M_m(\Gamma)+g(\Gamma)-1,
\end{align*}
since
$\supp{(\mathfrak c'_{f})} \cap \supp{(\mathfrak c'_{h})}=\emptyset $. The case $W\subset S_m(\Gamma)$ is treated 
similarly. 
\end{proof}

\section{Application to modular equation}\label{appmod}

In this section we complete the proof of Theorem \ref{ithm} following the approach explained in the introduction.
We start the proof with the following well--known lemma:

\begin{Lem}\label{appmod-1}
Assume that $\Gamma=SL_2(\mathbb Z)$. Then, we have the following:
\begin{align*}
& \text{div}{(\Delta)}=\mathfrak a_\infty\\
&\text{div}{(E_4)} =\frac13 \mathfrak a_{(1+\sqrt{-3})/2}\\
&\text{div}{(E^3_4)} = \mathfrak a_{(1+\sqrt{-3})/2}.
\end{align*}
\end{Lem}
\begin{proof} By Lemma \ref{prelim-1} (iv), we get $\text{deg}(\text{div}{(\Delta)})=1$. Since $\Delta$ is a cusp form, the
first formula follows. Again, by Lemma \ref{prelim-1} (iv), we get 
$\text{deg}(\text{div}{(E_4)})=1/3$. Set $\epsilon=(1+\sqrt{-3})/2$ and 
$\gamma=\left(\begin{matrix}1 & -1\\ 1& 0\end{matrix}\right)$. Then $\gamma.\epsilon=\epsilon$, and 
we have 
$$
E_4(\epsilon)=E_4(\gamma.\epsilon)=
j(\gamma, \epsilon)^{4} E_4(\epsilon)\implies E_4(\epsilon)=0,
$$
since $j(\gamma, \epsilon)^{4}=\epsilon^4=-\epsilon\neq 1$. Thus, by the preliminary considerations in Section 
\ref{prelim}, we obtain the second formula. The third formula is a direct consequence of the second.
\end{proof}

\vskip.2in 
In the remainder of this section we let  $\Gamma=\Gamma_0(N)$ and assume that $N\ge 2$. 
We warn the reader that we compute $\text{div}$ with respect to $\Gamma_0(N)$,  not with respect to 
$SL_2(\mathbb Z)$ as in the previous lemma.

Applying the arguments from the proof of (\cite{Miyake}, Theorem 4.2.7), we find that 
the representatives for $\Gamma_0(N)$--orbits of cusps for $\Gamma_0(N)$ are of the form 
$p/q$, $p, q\in \mathbb Z$ relatively prime, with exactly $\varphi((k, N/k))$ of them  satisfying $(q, N)=k$, for each 
$1\le k\le N$, $k| N$. When $k=N$, there is only one representative, and it belongs to the orbit 
$\Gamma_0(N).\infty$. We denote by $C_N$ the set of those representatives.
An elementary computation shows that  

\begin{align*}
SL_2(\mathbb Z)_{p/q}&=\left\{\left(\begin{matrix}\epsilon -pqt & p^2t\\ -q^2 t & \epsilon +pqt \end{matrix}\right): \ \
\epsilon=\pm 1, \ t\in \mathbb Z\right\}\\
\Gamma_0(N)_{p/q}&=\left\{\left(\begin{matrix}\epsilon -pqt & p^2t\\ -q^2 t & \epsilon +pqt \end{matrix}\right): \ \
\epsilon=\pm 1, \ t\in \mathbb Z, \ N| q^2t\right\}.
\end{align*}
Select $p', q'\in\mathbb Z$ such that $pp'+qq'=1$. Let $\sigma_{p/q}=
\left(\begin{matrix}p' & q'\\ -q & p \end{matrix}\right)\in SL_2(\mathbb R)$. Then, 
$\sigma_{p/q}.p/q=\infty$, and 
\begin{equation}\label{cccusppp}
\begin{aligned}
\sigma_{p/q}SL_2(\mathbb Z)_{p/q}\sigma^{-1}_{p/q} &=\left\{\pm 1\right\}\left\{
\left(\begin{matrix}1 & t\\ 0 & 1 \end{matrix}\right): \ \ t\in \mathbb Z\right\}\\
\sigma_{p/q}\Gamma_0(N)_{p/q}\sigma^{-1}_{p/q}&=\left\{\pm 1\right\}\left\{
\left(\begin{matrix}1 & t\\ 0 & 1 \end{matrix}\right): \ \ t\in \mathbb Z, \ N| q^2t\right\}, 
\end{aligned}
\end{equation}
where  the very last group is isomorphic to $\mathbb Z$ and generated by $N/(N, k^2)=
N/(k\cdot (k, N/k))$.

Next, we let 
$\Gamma'_0(N)$ to be the subgroup of $SL_2(\mathbb Z)$ consisting of all matrices 
$\left(\begin{matrix}* & a\\ * & * \end{matrix}\right)$ where $a$ is divisible by $N$.
If we let $\tau=\left(\begin{matrix}N & 0\\ 0 & 1 \end{matrix}\right)$, then 
$\Gamma'_0(N)=\tau \Gamma_0(N) \tau^{-1}$.
So, representatives of cusps for this group are $\tau.p/q=Np/q$. Obviously, we have 
$\Gamma'_0(N)_{Np/q}= \tau\Gamma_0(N)_{p/q} \tau^{-1}$. We let $\tau_{p/q}=\sigma_{p/q}\tau^{-1}$.
Using conjugation by $\left(\begin{matrix}0 & -1\\ 1 & 0 \end{matrix}\right)\in SL_2(\mathbb Z)$, the group 
$\Gamma'_0(N)$ is transformed onto $\Gamma'_0(N)$, and cusp $Np/q$ is transformed onto the cusp $-q/Np$ which is 
$\Gamma_0(N)$--equivalent to one of the form $p_1/q_1$, $p_1, q_1, \mathbb Z$ relatively prime, $(q_1, N)=N/k$, 
described above.

\begin{Lem}\label{appmod-2}
Assume that  $\Gamma=\Gamma_0(N)$ and $N\ge 2$. Set $\epsilon=(1+\sqrt{-3})/2$. Then, 
by considering $\Delta, E^3_4\in M_{12}(\Gamma_0(N))$,   we have the following:
\begin{align*}
\text{div}{(\Delta)}&=\sum_{p/q\in C_N} \frac{N}{k} \frac{1}{(k, N/k)} \mathfrak a_{p/q}\\
\text{div}{(\Delta (N\cdot ))}&=\sum_{p/q\in C_N} k \frac{1}{(k, N/k)} \mathfrak a_{p/q}\\
\text{div}{(E^3_4)} &= \sum_{\gamma\in \Gamma_0(N) \backslash SL_2(\mathbb Z)/SL_2(\mathbb Z)_\epsilon}
m_\gamma \mathfrak a_{\gamma.\epsilon}\\
\text{div}{(E^3_4(N\cdot ))} &= \sum_{\gamma\in \Gamma_0(N) \backslash \{\frac1N \gamma.\epsilon; \  \gamma\in SL_2(\mathbb Z)\}} 
m_\gamma \mathfrak a_{\gamma.\epsilon},
\end{align*}
where $m_\gamma\ge 1$ in both cases. 
\end{Lem}
\begin{proof} We use the construction of the divisor of a modular form explained in Section \ref{prelim}.
To find Fourier expansion at $p/q$ for $\Delta$ considered
as a cusp form on $\Gamma_0(N)$, we use Lemma \ref{appmod-1} and  (\ref{cccusppp}). 
 The first expression in  (\ref{cccusppp})  and Lemma \ref{appmod-1} tell us that the Fourier 
 expansion at $p/q$ for $\Delta$ considered as a cusp form on $SL_2(\mathbb Z)$ has the first Fourier 
coefficient non--zero. We use the second expression in (\ref{cccusppp}) to scale formula. 
Similalrly, we first determine the divisor of $\Delta$ considered
as a cusp form on $\Gamma'_0(N)$ using the discussion in the paragraph before the statement of the lemma.
Finally, we get the divisor of  $\Delta (N\cdot )$ by using the remark from the end of Section \ref{prelim}

For the third and fourth formula, we observe that  
 Lemma \ref{appmod-1} implies that $SL_2(\mathbb Z)$--orbit of  $\epsilon$ are only zeroes of $E_4$. 
Some of them are elliptic for $\Gamma_0(N)$ and some are not. But, since elliptic among them are of order $3$ and $m=12$,
 we see that by Lemma \ref{prelim-1} (vi) we have $\mathfrak c'_{E^3_4} =\text{div}{(E^3_4)}$. 
This immediately  implies the third formula. For the fourth, only zeroes
of $E^3_4(N\cdot )$ are in the set $\{\frac1N \gamma.\epsilon; \  \gamma\in SL_2(\mathbb Z)\}$, and the claim follows.
The last claim is obvious from Lemma \ref{appmod-1}.
\end{proof}

\vskip .2in

Having Lemma \ref{appmod-2}, it is easy to complete the proof of Theorem \ref{ithm}. 
We consider the regular  map given by (\ref{mmap}). As we indicated in the introduction
this map is birational. Thus, by  Theorem \ref{p2-120}, we 
have the following formula for the degree of the corresponding curve which is the same 
as the degree of $N$--th modular polynomial $\Phi_N$
$$
\deg{(\Phi_N)}= \dim S_{24}(\Gamma_0) + g(X_0(N)) -1 - \sum_{\mathfrak a\in X_0(N)} 
\min{\left(\mathfrak c_{\Delta(N\cdot )\Delta}(\mathfrak a),  \mathfrak c_{E^3_4\Delta(N\cdot )}(\mathfrak a), 
\mathfrak c_{E^3_4(N\cdot ) \Delta}(\mathfrak a) \right)}.
$$

Now, let us write $\nu_2(\Gamma_0(N))$, $\nu_3(\Gamma_0(N))$, and $\nu_\infty(\Gamma_0(N))$  for the number of inequivalent 
elliptic points of order $2$, inequivalent elliptic points of order $3$, and inequivalent cusps for 
$\Gamma_0(N)$, respectively.

Next, Lemma \ref{appmod-2}, implies that 
\begin{align*}
&\sum_{\mathfrak a\in X_0(N)}
\min{\left(\mathfrak c_{\Delta(N\cdot )\Delta}(\mathfrak a),  \mathfrak c_{E^3_4\Delta(N\cdot )}(\mathfrak a), 
\mathfrak c_{E^3_4(N\cdot ) \Delta}(\mathfrak a) \right)}= \\
&=\sum_{\substack{k|N \\0<k\le N}}\varphi((k, N/k))\min{\left(k,  \frac{N}{k}\right)} 
\frac{1}{(k, N/k)}- \nu_\infty(\Gamma_0(N))\\
&=\varphi(\sqrt{n})+ 2\sum_{\substack{k|N \\\sqrt{N}<k\le N}}\varphi((k, N/k))
\frac{N/k}{(k, N/k)}- \nu_\infty(\Gamma_0(N)),
\end{align*}
where we use the convention from the introduction that $\varphi(\sqrt{n})=0$ if $n$ is not a perfect square.

Now,  inserting explicit formulas for $\dim S_{24}(\Gamma_0(N))$ 
(see Lemma \ref{prelim-1}(v)), using (see \cite{Miyake}, Theorem 4.2.11) 
\begin{equation}\label{gggenus}
g(\Gamma_0(N))=1+ \frac{1}{12}[SL_2(\mathbb Z): \Gamma_0(N)]- \frac{\nu_2(\Gamma_0(N))}{4}-
 \frac{\nu_3(\Gamma_0(N))}{3}- \frac{\nu_\infty(\Gamma_0(N))}{2},
\end{equation}
and recalling that elliptic points of $\Gamma_0(N)$ are of order $2$ or $3$, we obtain 
$$
\deg{(\Phi_N)}= 2[SL_2(\mathbb Z): \Gamma_0(N)] -\varphi(\sqrt{n})- 2\sum_{\substack{k|N \\\sqrt{N}<k\le N}}\varphi((k, N/k))
\frac{N/k}{(k, N/k)}.
$$

This formula can be further simplified if we compute the degree of the divisor of $\Delta$ in two ways: first using
 its definition (see Section \ref{prelim}; using the discussion before Lemma \ref{appmod-2}), and using 
Lemma \ref{prelim-1} (iv) which gives us by means of (\ref{gggenus})
 
\begin{align*}
\text{deg}(\text{div}{(\Delta)})&=12(g(\Gamma_0(N))-1)+ 6\nu_\infty(\Gamma_0(N)) + 3\nu_2(\Gamma_0(N)) + 4\nu_3(\Gamma_0(N))
\\
&= [SL_2(\mathbb Z): \Gamma_0(N)].
\end{align*}
As a result we obtain the following identity\footnote{In passing, we remark that (\cite{Miyake}, Theorems 4.2.5) 
implies $[SL_2(\mathbb Z): \Gamma_0(N)]=\psi(N)$ (see the 
Introduction).}
$$\sum_{\substack{k|N \\0<k\le N}} \frac{N}{k}  \frac{1}{(k, N/k)} \varphi((k, N/k))=[SL_2(\mathbb Z): \Gamma_0(N)].
$$
 Now,  Theorem \ref{ithm} follows easily.

\section{Existence of Integral models}\label{exist}

The goal of this section is to prove the following corollary to Theorem \ref{p2-thm} (stated as Theorem \ref{ithm-3}
in the Introduction):

\begin{Cor}\label{exist-1}
Assume that $N\not\in \{1, 2, 3, 4, 5, 6, 7, 8, 9, 10, 12, 13, 16, 18, 25\}$ (so that the genus 
$g(\Gamma_0(N))\ge 1$). Assume that  $m\ge 4$ (if $N\neq 11$) and $m\ge 6$ (if $N=11$) is an even integer. 
Let $f, g \in S_m(\Gamma_0(N))$  be linearly independent with integral $q$--expansions.
Then, there exists infinitely many   $h\in S_m(\Gamma_0(N))$   
with integral $q$--expansion such that  we have the following:
\begin{itemize}
\item[(i)] $X_0(N)$ is birationally equivalent to  $\cal C(f, g, h)$, 
\item[(ii)] $\cal C(f, g, h)$ has degree equal to $\dim S_m(\Gamma_0(N))+g(\Gamma_0(N))-1$ 
(this number can be easily explicitly computed using Lemma \ref{prelim-1}(v) and (\ref{gggenus})); 
if $N=11$, then the minimal possible degree achieved (for $m=6$) is $4$, and if $N\neq 11$, then the minimal 
possible degree achieved (for $m=4$) is
$$
\frac13 \psi(N)-  \frac13 \nu_3 -\sum_{d>0, d|N} \phi((d, N/d)),
$$
where $\nu_3$ is the number of elliptic elements of order three on $X_0(N)$,
$\nu_3=0$ if $9|N$, and $\nu_3=\prod_{p|N}\left(1+ \left(\frac{-3}{p}\right)\right)$ otherwise.

\item[(iii)] the equation of $\cal C(f, g, h)$ has integral coefficients.
\end{itemize}
\end{Cor}
\begin{proof} First, by Eichler--Shimura theory, for 
each even integer $m\ge 2$ the space of cusp forms $S_m(\Gamma)$ has a basis as a complex vector space 
consisting of forms which have integral $q$--expansions. So, if we 
$f$ and $g$ with  integral coefficients in their $q$--expansions, then we can select infinitely many $h$ which also
 have integral coefficients in  their $q$--expansions. This is because $\mathbb Z^l$ and a complement of finite subset of 
it are Zariski dense in $\mathbb C^l$, for any $l\ge 1$. As a consequence, since the polynomial 
equation of $\cal C(f, g, h)$, after inserting the $q$--expansions for $f$, $g$, and $h$,  
produces a homogeneous system with integral coefficients which has the coefficients of the polynomial as a unique 
solution up to a scalar, the coefficients can be taken to be integral as well. At the end, to apply Theorem \ref{p2-thm},
so that above discussion is valid, we need to assure that $\dim S_m(\Gamma_0(N))\ge \max{(g(\Gamma_0(N))+2, 3)}$.
Since we assume that $g(\Gamma_0(N))\ge 1$, we only require that $\dim S_m(\Gamma_0(N))\ge g(\Gamma_0(N))+2$.
Using  Lemma \ref{prelim-1}(v),  and the notation introduced at the end of Section \ref{appmod},  we obtain

\begin{multline*}
\dim S_m(\Gamma_0(N)) =
(m-1)(g(\Gamma_0(N))-1)+\left(\frac{m}{2}-1\right)\nu_\infty(\Gamma_0(N)) +\\
+ \left[\frac{m}{4}\right]\nu_2(\Gamma_0(N))
+\left[\frac{m}{3}\right]\nu_3(\Gamma_0(N)).
\end{multline*}

By (\cite{Miyake}, Theorem 4.2.7), we have 
$$\nu_\infty(\Gamma_0(N))=\sum_{d>0, d|N} \phi((d, N/d))\ge 3
$$
unless $N$ is prime number in which case $\nu_\infty(\Gamma_0(N))=2$. Next, unless 
$\nu_2(\Gamma_0(N))=\nu_3(\Gamma_0(N))=0$, above formula shows that  for $m=4$ we have 
$$
\dim S_4(\Gamma_0(N))\ge 3(g(\Gamma_0(N))-1)+ 2\left(\frac{4}{2}-1\right)+ 1=3 g(\Gamma_0(N)
\ge g(\Gamma_0(N))+2, 
$$
since we assume that $g(\Gamma_0(N)\ge 1$. Similarly we have if $\nu_2(\Gamma_0(N))=\nu_3(\Gamma_0(N))=0$ but 
$N$ is not prime. It remains to consider the case $N$ is prime and  $\nu_2(\Gamma_0(N))=\nu_3(\Gamma_0(N))=0$. 
In this case 
$$
\dim S_4(\Gamma_0(N))= 3g(\Gamma_0(N))-1\ge g(\Gamma_0(N))+2 
$$ 
if and only if $g(\Gamma_0(N))\ge 2$. It remains to consider the case 
 $N$ is prime, $\nu_2(\Gamma_0(N))=\nu_3(\Gamma_0(N))=0$, and $g(\Gamma_0(N))=1$. In this case (\ref{gggenus}) 
gives us $[SL_2(\mathbb Z): \Gamma_0(N)]=12$. Applying (\cite{Miyake}, Theorem 4.2.5), we see that 
$\psi(N)=N+1=12$ since $N$ is prime. Hence, $N=11$. In this case, we use (\cite{Miyake}, Theorem 4.2.5) to check that we 
indeed have $\nu_2(\Gamma_0(11))=\nu_3(\Gamma_0(11))=0$ and $g(\Gamma_0(N))=1$. This gives us 
$\dim S_4(\Gamma_0(11)) = 2$ and $\dim S_6(\Gamma_0(11)) =4$. In this case is 
$\dim S_6(\Gamma_0(11)) + g(\Gamma_0(11))-1=4$. 

Apart from this case, we can use $m=4$, which gives us the formula

\begin{align*}
\dim S_4(\Gamma_0(N)) + g(\Gamma_0(N))-1&=
\frac13 \psi(N)-  \frac13 \nu_3(\Gamma_0(N))- \nu_\infty(\Gamma_0(N)).
\end{align*}
Now, we apply  (\cite{Miyake}, Theorem 4.2.5) to complete the proof of the corollary.
\end{proof}

\section{An Improvement of Theorem \ref{p2-thm}} \label{exam}

The goal of this section is to prove a corollary to Theorem \ref{p2-thm} (stated as Theorem \ref{ithm-3}
in the Introduction) and its improvement. We prove necessary lemmas first. 

\begin{Lem}\label{exam-1} Let $N\ge 2$. Then, 
$\Delta, E^3_4, \Delta(N\cdot ),$ and $E^3_4(N\cdot )$ are linearly independent.
\end{Lem}
\begin{proof} We write their $q$--expansions
\begin{align*}
\Delta&=q-24q^2+252q^3+\cdots\\
E^3_4 &=1+720q+172800 q^2+ 13824000q^3+\cdots\\
\Delta(N\cdot )&=q^N-24q^{2N}+252q^{3N}+\cdots\\
E^3_4(N\cdot )&=1+720q^N +172800 q^{2N}+ 13824000q^{3N}+\cdots.
\end{align*}
Now, assume that  for $\alpha, \beta, \gamma, \delta\in \mathbb C$ we have 
$$
\alpha\Delta+ \beta E^3_4 +\gamma \Delta(N\cdot )+ \delta E^3_4(N\cdot )=0.
$$
Inserting $q$--expansions in that expression, the coefficients of $q^0$ and $q$ must be equal to zero i.e., 
$\beta+\delta=0$ and $\alpha +720\beta=0$. If $N>2$, then $-24\alpha+172800 \beta=0$ is a coefficient of $q^2$.
 If $N=2$, then $252\alpha+ 13824000\beta=0$ is a coefficient of $q^3$. In either case, 
 we get immediately $\alpha=\beta=\gamma=\delta=0$.
\end{proof}

\begin{Lem}\label{exam-2} Let $N\ge 2$. Then, by considering $\Delta$ and $E^3_4$ as modular forms of $\Gamma_0(N)$, 
we have $\supp{(\mathfrak c'_{\Delta})} \cap \supp{(\mathfrak c'_{E^3_4})}=\emptyset $. In particular, a four--dimensional 
subspace $W\subset M_{12}(\Gamma_0(N))$ spanned by $\Delta, E^3_4, \Delta(N\cdot ),$ and $E^3_4(N\cdot )$ 
separates points on $X_0(N)$.
\end{Lem}
\begin{proof} We use the expressions for $\text{div}{(\Delta)}$ and $\text{div}{(E^3_4)}$ from Lemma \ref{appmod-2}, 
and the definition of $\mathfrak c'_{\Delta}$ and $\mathfrak c'_{E^3_4}$ Lemma \ref{prelim-1} (vi)
to see that $\supp{(\mathfrak c'_{\Delta})} \cap \supp{(\mathfrak c'_{E^3_4})}=\emptyset $.
\end{proof}

\vskip .2in

\begin{Cor}\label{exam-3} Let $N\ge 2$. 
Then, there exists infinitely many  $(\alpha, \beta, \gamma, \delta)\in \mathbb Z^4$ such that 
$X_0(N)$ is birational with 
$\cal C(\Delta,  E^3_4,  \alpha\Delta+ \beta E^3_4 +\gamma \Delta(N\cdot )+ \delta E^3_4(N\cdot ))$,
and 
$$
\deg{\cal C(\Delta,  E^3_4,  \alpha\Delta+ \beta E^3_4 +\gamma \Delta(N\cdot )+ \delta E^3_4(N\cdot ))}=\psi(N).
$$
\end{Cor}
\begin{proof} We observe that $W\not\subset S_{12}(\Gamma_0(N))$. Next, $W$ separates points on $X_0(N)$ by Lemma
\ref{exam-2}. It also generates the field of rational functions on $X_0(N)$ (see the Introduction).
This shows that we can apply  Theorem \ref{p2-thm}. We just observe that $\mathbb Z^4$ is Zariski dense in 
 $\mathbb C^4$, and that the same holds for a complement of a  finite subset in $\mathbb Z^4$.
The curves are of degree $\dim M_{12}(\Gamma_0(N))+ g(\Gamma_0(N))-1$. 
Finally, we compute $\dim M_{12}(\Gamma_0(N))+ g(\Gamma_0(N))-1$ using similar computations to those made  
at the end of Section \ref{appmod}. We leave details to the reader.
\end{proof}

\vskip .2in

The following result is an improvement of Theorem \ref{p2-thm}:
 
\begin{Thm}\label{exam-4} Let $N\ge 2$. 
Then, there exists infinitely many  pairs $(\alpha, \beta)\in \mathbb Z^2$ such that 
$X_0(N)$ is birational with 
$\cal C(\Delta,  E^3_4,  \alpha \Delta(N\cdot )+ \beta E^3_4(N\cdot ))$,
and 
$$
\deg{\cal C(\Delta,  E^3_4,  \alpha \Delta(N\cdot )+ \beta E^3_4(N\cdot ))}=\psi(N).
$$
\end{Thm}
\begin{proof} For $W$ defined by Lemma \ref{exam-2}, we let 
$f_0=\Delta$,  $f_1=E^3_4$,  $f_2=\Delta(N\cdot )$, $f_3= E^3_4(N\cdot )$. This is a basis of $W$ in the notation of 
the proof of Lemma \ref{p2-8}. In our case the resultant $R(\lambda_2, \lambda_3)$ from the proof of Lemma \ref{p2-8}
is used to secure that $(\lambda_2f_2+\lambda_3f_3)/f_0$ is a primitive element of appropriate extension if and only if 
$R(\lambda_2, \lambda_3)\neq 0$ (see (\ref{p2-50})). The remainder of the proof of  Lemma \ref{p2-8} is not relevant 
for us since we have the first claim of Lemma \ref{exam-2} at our disposal: $\supp{(\mathfrak c'_{\Delta})} \cap \supp{(\mathfrak c'_{E^3_4})}=\emptyset $. With this in hand, the proof of Theorem \ref{p2-thm} is easy to modify so that with the help
of arguments in the proof of  Corollary \ref{exam-4} we can complete the proof.
\end{proof}

\end{document}